\documentclass[11pt,a4paper]{article}

\usepackage{amsmath, amsthm, amsfonts, amssymb, mathrsfs}
\usepackage{bbm}
\usepackage{tikz}
\usetikzlibrary{matrix,arrows}
\usepackage{bm}
\usepackage{subfig}
\usepackage{listings}
\usepackage{pifont}

\newtheorem{thm}{Theorem}[section]

\newtheorem{lemma}[thm]{Lemma}
\newtheorem{proposition}[thm]{Proposition}

\theoremstyle{definition}
\newtheorem{definition}[thm]{Definition}
\newtheorem{remark}[thm]{Remark}

\newcommand{\mx}{\mathrm{max}}
\newcommand{\mn}{\mathrm{min}}
\newcommand{\nx}{\mathrm{2max}}
\newcommand{\lx}{\mathrm{2min}}

\newcommand{\ol}[1]{\overline{#1}}

\newcommand{\CC}{\mathcal{C}}
\newcommand{\C}{\mathbbm{C}}
\newcommand{\R}{\mathbbm{R}}
\newcommand{\E}{\mathbbm{E}}

\providecommand{\mathbold}[1]{\bm{#1}}
\newcommand{\vct}[1]{\mathbold{#1}}
\newcommand{\IN}{\mathbbm{N}}
\newcommand{\IR}{\mathbbm{R}}
\newcommand{\polar}{\circ}
\newcommand{\veps}{\varepsilon}
\newcommand{\sdim}{\delta}
\newcommand{\sdimw}{w}
\DeclareMathOperator{\Prob}{Prob}
\newcommand{\Expect}{\operatorname{\mathbb{E}}}
\DeclareMathOperator{\vol}{vol}

\DeclareMathOperator{\dist}{dist}
\DeclareMathOperator{\Uniform}{U}
\DeclareMathOperator{\codim}{codim}
\newcommand{\resTMP}[2]{#1\to#2}

\newcommand{\nres}[3]{\|\vct{#1}\|_{\resTMP{#2}{#3}}}
\newcommand{\sres}[3]{\sigma_{\resTMP{#2}{#3}}(\vct{#1})}
\newcommand{\srestm}[3]{\sigma_{\resTMP{#2}{#3}}(-\vct{#1}^T)}

\newcommand{\Ren}{\mathcal{R}}
\newcommand{\RCD}[3]{\Ren_{#2,#3}(\vct{#1})}

\newcommand{\mtx}[1]{\mathbold{#1}}
\newcommand{\norm}[1]{\Vert {#1}\Vert}
\newcommand{\Proj}{\ensuremath{\mtx{\Pi}}}
\newcommand{\lspan}{\mathrm{span}}
\newcommand{\D}{\mathcal{D}}
\renewcommand{\P}{\mathcal{P}}

\newcommand{\CP}{\mathbb{CP}}

\usepackage{tabu}
\usepackage[pdftex, pdfpagelabels]{hyperref}
\hypersetup{
  plainpages=false,
  colorlinks,
  citecolor=black,
  filecolor=black,
  linkcolor=black,
  urlcolor=blue,
  pdfpagemode = UseNone,
  pdfstartview={XYZ null null 0.75},
}

\begin{document}

\title{Average-case complexity without the black swans}

\author{Dennis Amelunxen \and Martin Lotz}

\maketitle

\begin{abstract}
We introduce the concept of weak average-case analysis as an attempt to achieve theoretical complexity results that are closer to practical experience than those resulting from traditional approaches. This concept is accepted in other areas such as non-asymptotic random matrix theory and compressive sensing, and has a particularly convincing interpretation in the most common situation encountered for condition numbers, where it amounts to replacing a null set of ill-posed inputs by a ``numerical null set''.
We illustrate the usefulness of these notions by considering three settings: (1)~condition numbers that are inversely proportional to a distance of a homogeneous algebraic set of ill-posed inputs; (2)~the running time of power iteration for computing a leading eigenvector of a Hermitian matrix; (3)~Renegar's condition number for conic optimisation.
\end{abstract}

\section{Introduction}
Depending on context and tradition,
a computational problem can mean something practical that begs to be solved as efficiently as possible, or a mathematical object in its own right, to be analysed, classified, and understood. In the first sense, the aim is to develop methods that work well on problems of interest, while in the second, complexity-theoretic sense, algorithms are merely devices used to show that a problem can be solved within certain resource constraints, e.g., in a certain complexity class or with a running time bounded by some function of the input size. Needless to say, complexity-theoretic results are often only weakly correlated with practical experience; a typical example is the simplex method. This is particularly true for numerical problems, where often a condition number serves as a proxy to computational complexity. In this note we aim at shortening the gap between complexity results and practical experience in a theoretically sound way.

In situations where worst-case analysis is meaningless or overly pessimistic, an established practise in complexity theory is to endow the space of inputs with a probability measure and then analyze random variables of interest, like the condition number, induced by this measure. A major point of discussion is the explanatory power of the random model, which needs to satisfy some assumptions to be within reach of a theoretical analysis, in comparison with ``typical problems'' that are encountered in practice.
In this article we do not address the issue of the accuracy of the chosen probability model for the input, but rather a different point of discussion which has received little attention so far. Because even in the most extreme (and unlikely) case that the chosen probability model is 100\% accurate, one may \emph{experience} a behavior that is not predicted by the traditional method of analysis. This discrepancy is due to ``black swans''; inputs that dominate the theoretical analysis, but which are at the same time extremely rare, so that they practically never show up. For a very concrete example, the expected condition number of a random quadratic Gaussian matrix is infinite, yet most matrices are well-conditioned. On a side note, this discrepancy might also be the reason for a claim by
Goldstine and von Neumann~\cite[p.14]{von1963john}
(also pointed out by Edelman and Rao~\cite{edelman2005random}) that ``for a random matrix of order $n$ the expectation value of $l$ [the condition number] has been shown to be about $n$.'' 

In our approach we allow to discard a small subset of the input space. This more liberal attitude towards accounting for all inputs reflects modern practice better, as seen in convex relaxation methods such as compressive sensing, where it is entirely acceptable that an algorithm may even fail on exponentially small sets. From a numerical point of view, we are not able to distinguish between null sets and very small sets due to round-off errors, so the latter may be considered as ``numerical null sets''. As we will see, disregarding such a practically invisible sets in the analysis can lead to dramatically improved bounds.

\begin{definition}
For $k\in\IN$ let $(M_k,\mu_k)$ be a probability space and let $T_k\colon M_k\to\IR$ be a $\mu_k$-measurable function. We say that the family $\{T_k\}$ has a \emph{weak expectation} of $O(f(k))$ if there exists a family of sets of \emph{exceptional inputs}, $E_k\subseteq M_k$, such that $\mu_k(E_k)=e^{-\Omega(k)}$ and the conditional expectation, conditioned on the nonexceptional inputs, $\E[T_k(\vct x)\mid \vct x\not\in E_k]$ is bounded by $f(k)$.
\end{definition}

Accordingly, we will speak of \emph{weak average-case complexity} or \emph{weak smoothed complexity} of algorithms or condition numbers. We use the term `weak' because every traditional complexity analysis without exceptional inputs is in particular a corresponding weak analysis; when it comes to the informative value of the analysis one may in fact argue that the weak analysis is ``stronger'' than the traditional one. Of course, as with every other use of the $O$-calculus, one has to be careful about the involved constants, but it seems natural, if not unavoidable, to make the above definition independent of these constants. We illustrate this analysis on three types of problems:
\begin{enumerate}
 \item Condition numbers inversely proportional to a distance to a homogeneous algebraic set of ill-posed inputs;
 \item The running time of power iteration for computing a leading eigenvector of a Hermitian matrix;
 \item Renegar's condition number for conic optimisation.
\end{enumerate}

In all three cases the considered expectations over the whole input space are infinite; taking out exponentially small subsets instead yield polynomial conditional expectations, and in the third case even a constant conditional expectation.

\subsection{Conic condition numbers}
Many condition numbers in numerical analysis 
come in form of an inverse normalized distance to a set of ill-posed inputs,
  \[ \CC(\vct x)\approx \frac{\|\vct x\|}{\dist(\vct x,\Sigma)} , \]
where $\Sigma$ is invariant under nonnegative scaling. The prototypical example is the matrix condition number, where $\Sigma$ is the set of singular matrices, but also condition numbers for eigenvalue computation and for systems of polynomial equations fall into this framework; see~\cite{BCL:08} for some examples and \cite{Condition} for a general discussion.
Assuming such a geometric characterisation, then in the case where $\Sigma$ is well-approximated by a set of real codimension $1$, the correct asymptotic order of the tail bounds for large values of~$t$ is
  \[ \Prob\{\CC(\vct x)\geq t\} \approx t^{-1}\,\vol\Sigma, \]
where $\vol\Sigma$ is some natural measure of a projective version of $\Sigma$. A direct consequence is
 \[ \E[\CC(\vct x)] = \infty . \]
This certainly does not reflect the experience that the condition for random inputs is usually seen as unproblematic. 
An explanation for this codimension conundrum may be found in the simple illustration in Figure~\ref{fig:influence-Sigma}: although~$\Sigma$ does have codimension one in both pictures, it influences in the second case only a small part of the input space in this way; the predominant part of inputs perceives $\Sigma$ as ``smaller''.

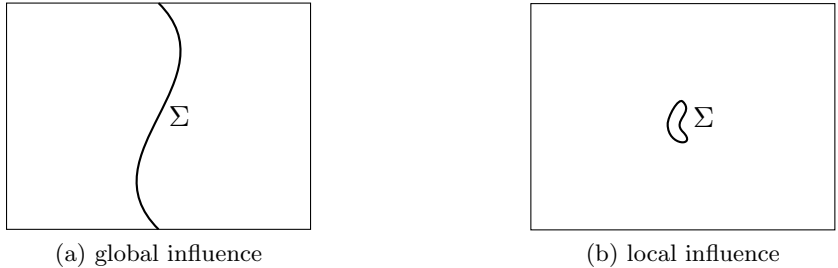
\begin{figure}[!ht]
  \begin{center}
  \subfloat[global influence]{
    \begin{tikzpicture}
    \draw (0,0) rectangle (4,3);
    \draw[thick] (2,0) .. controls (1,1) and (3,2) .. (2,3);
    \path (2,1.5) node[right]{$\Sigma$};
    \end{tikzpicture}
  }
  \rule{25mm}{0mm}
  \subfloat[local influence]{
    \begin{tikzpicture}
    \draw (0,0) rectangle (4,3);
    \draw[thick] (2,1.3) .. controls ++(-0.15,0.15) and ++(0.15,-0.15) .. (2,1.7)
                 (2,1.7) .. controls ++(-0.06,0.06) and ++(-0,0.1) .. (1.8,1.4)
                 (1.8,1.4) .. controls ++(0,-0.3) and ++(0.2,-0.2) .. (2,1.3);
    \path (2,1.5) node[right]{$\Sigma$};
    \end{tikzpicture}
  }
  \end{center}
  \caption{Illustration of the possible influences of ill-posed inputs.}
  \label{fig:influence-Sigma}
\end{figure}

The result of such a situation is that the small part around the set of ill-posed inputs \emph{dominates the complexity analysis}. And this part of local points around $\Sigma$ may in fact be \emph{exponentially small}. 

\begin{remark}
If one is only interested in the log of the condition, then this is not a big issue, and one may derive again general bounds for the expected condition of a slightly perturbed input, see~\cite{BCL:08} and~\cite[Ch.~22]{Condition}. 
\end{remark}

We adopt the setup of~\cite{BCL:08}: let $\Sigma\subseteq\IR^{n+1}$ be invariant under positive scaling and let the condition number $\CC\colon S^n\to\IR$ be {\em defined} by $\CC(\vct x)=\|\vct x\|/\dist(\vct x,\Sigma)$, where $\dist(\vct x,\Sigma)=\inf_{\vct y\in \Sigma}\{\|\vct x-\vct y\|\}$. Such a condition number is called a \emph{conic condition number}. For the smoothed analysis setting we denote for a point $\vct z\in S^n$ and $\sigma>0$ the ball, i.e., the spherical cap, in $S^n$ of radius $\sigma$ around $\vct z$ by $B(\vct z,\sigma)$.

\begin{thm}\label{thm:con-cond-no.}
Let $\Sigma\cap S^n\subseteq W$, where $W\neq S^n$ is the zero set in $S^n$ of homogeneous polynomials of degree at most $d\geq1$ and let $\sigma\in(0,1]$. If $\codim\Sigma=1$ and $\vct z\in \Sigma$, then for $\vct x\in B(\vct z,\sigma)$ uniformly at random, $\E[ \CC(\vct x)]=\infty$. On the other hand, regardless of the codimension, for all $\vct z\in S^n$ there exists a set $E_{\vct z}\subseteq S^n$ such that for $\vct x\in B(\vct z,\sigma)$ uniformly at random,
\begin{align*}
   \Prob\{\vct x\in E_{\vct z}\} & < e^{-n} , & \E\big[ \CC(\vct x)\mid \vct x\not\in E_{\vct z}\big] & < \frac{13 d n(n+1)}{(1-e^{-n})\sigma} = O\Big(\frac{d n^2}{\sigma}\Big) .
\end{align*}
\end{thm}

We note that the bounds in Theorem~\ref{thm:con-cond-no.} are still rather coarse when applied in different settings, the reason being that they are quite general. For example, for a version $\kappa_F(\mtx{A})$ of the matrix condition number (using the Frobenius norm) we get from the above,
\begin{equation*}
 \Expect[\kappa_F(\mtx{A}) | \mtx{A}\not\in E] = O\big(n^{4}/\sigma\big).
\end{equation*}
Once the derivation is understood, it is straight-forward to apply the weak average-case analysis to more precise, problem specific, bounds, such as the smoothed analysis of the matrix condition number by Wschebor~\cite{Wsch:04}.

\subsection{Power iteration}\label{ssec:powerit}
Power iteration is a classic method for computing a dominant eigenvector and eigenvalue of a matrix.
Let $\mtx{A}\in \C^{n\times n}$ be a Hermitian matrix (this also works for other matrices, but for simplicity we restrict to this case) with eigenvalues $\lambda_i$, ordered according to their absolute values, which are the singular values of the matrix,
\begin{equation*}
 |\lambda_1|\geq |\lambda_2|\geq  \cdots \geq |\lambda_n|.
\end{equation*}
Let $\vct{u}_1,\dots,\vct{u}_n$ be eigenvectors corresponding to this ordering of eigenvalues. The eigenvalue $\lambda_1$ is called a dominant eigenvalue, while $\vct{u}_1$ is called a dominant eigenvector. The power iteration generates a sequence of unit vectors $\vct{p}_k$, $k\geq 0$, by setting
\begin{equation*}
 \vct{p}_k = \frac{\mtx{A}\vct{p}_{k-1}}{\norm{\mtx{A}\vct{p}_{k-1}}} = \frac{\mtx{A}^k\vct{p}_0}{\norm{\mtx{A}^k\vct{p}_0}}.
\end{equation*}
For a generic starting point, the algorithm converges geometrically with ratio $|\lambda_1|/|\lambda_2|$, see~\cite[9.2]{wilkinson1965algebraic} or~\cite[7.3]{GoLoan} for an analysis. As there are degenerate cases in which the algorithm does not converge at all, or very slowly, it is natural to ask about the average-case complexity of power iteration, where the average is taken over both the starting points and the space of input matrices. Such an analysis was carried out by Kostlan~\cite{K:88,K:91}. In~\cite{K:88}, he analyzes the power method for matrices from the Gaussian orthogonal and unitary ensembles: the expected number of steps is infinite (Thm.~4.3); but in Thm.~4.4 he also proves a result akin to a weak average-case analysis, which he calls `generalized average' (a notion he attributes to Smale), where he takes out a subset of the input space of measure~$\eta$. Unfortunately, setting this measure exponentially small, $\eta=e^{-\Omega(n)}$, as required for a weak average-case analysis, only yields an exponential bound on the number of iterations. In~\cite{K:91}, Kostlan discusses an iterated squaring method that understandably reduces the complexity of the method, but of course this is impractical for applications. The purpose of our work is to improve Kostlan's analysis by showing that power iteration has polynomial weak average running time.

To quantify the convergence, we use the Fubini-Study metric (a.k.a. angle)
\begin{equation*}
 d_R(\vct{x},\vct{y}) = \arccos \left(\frac{|\langle \vct{x},\vct{y}\rangle |}{\|\vct{x}\|\|\vct{y}\|}\right),
\end{equation*}
where $\|\cdot\|$ is the norm induced by the Hermitian inner product on $\C^n$.
For $0<\alpha\leq\pi/2$ and $\vct{x}\in \C^{n}$, let $\rho_{\alpha}(\vct{A},\vct{x})$ be the minimum number of iterations of the power method with $\vct{p}_0=\vct{x}$ that bring $\vct{p}_k$ into an $\alpha$-neighbourhood of $\vct{u}_1$, 
\begin{equation*}
 \rho_{\alpha}(\vct{A},\vct{x}) = \min\{k \mid d_R(\vct{p}_k,\vct{u}_1)\leq \alpha\}.
\end{equation*}
Define the expected value
\begin{equation*}
 \rho_{\alpha}(\mtx{A}) = \underset{\vct{x}\sim \Uniform(\CP^{n-1})}\Expect [\rho_{\alpha}(\mtx{A},\vct{x})],
\end{equation*}
where $\Uniform(\CP^{n-1})$ denotes the uniform distribution on complex projective space (the notion $\rho_{\alpha}(\mtx{A},\vct{x})$ is well-defined on $\CP^{n-1}$). In what follows, we assume as random model the Gaussian Unitary Ensemble (GUE). A GUE matrix is defined by $\mtx{H} = \frac{1}{2}(\mtx{G}+\mtx{G}^*)$, where $\mtx{G}$ is a complex Gaussian matrix, whose entries have standard Gaussian real and imaginary parts. In particular, the off-diagonal entries have real and imaginary part with mean $0$ and variance $1/2$, while the diagonal entries are real with mean $0$ and variance $1$.

\begin{thm}\label{thm:powerit-main}
Let $\alpha\in (0,\pi/4)$, $n\geq 1$, and $\mtx{H}\in \C^{n\times n}$ a GUE matrix. Then there exists a set $E_n\subset \C^{n\times n}$ of exceptional inputs, such that
\begin{equation*}
  \Prob\{\mtx{H}\in E_n\}\leq e^{-n},
\end{equation*}
and
\begin{equation*}
  \Expect[\rho_\alpha(\mtx{H}) \mid \mtx{H}\not\in E_n] < \frac{p(n)}{1-e^{-cn}}\, \log \cot\alpha , \quad p(n) = O(n^2)
\end{equation*}
for some constant $c>0$.
In particular, power iteration runs in weak polynomial time.
\end{thm}

The result rests on the analysis of the quotient of the largest two singular values of a GUE matrix.
It would be interesting to see to what extent this result generalizes to other random models.
 The results are likely to carry over to Wigner matrices with sub-Gaussian entries. In this setting, the task would amount to finding good small ball probability bounds on the ratio or sum of extreme eigenvalues. A more challenging, though also more practically relevant, problem would be to study the ratio of singular values for a stochastic model of sparse and structured matrices, such as those that might arise in the discretisation of partial differential equations with random coefficients.

\subsection{Renegar's condition number and biconic feasibility}
The primal and dual (homogeneous) feasibility problems with reference cone~$C\subseteq\IR^m$ are the decision problems
\\
\def\tmpX{2mm}
\begin{minipage}{0.45\textwidth}
\begin{align}
   \exists \vct x & \in C\setminus\{\vct0\} \quad\text{s.t.} \hspace{\tmpX} \vct{Ax}=\vct0 ,
\tag{P}
\end{align}
\end{minipage}
\rule{0.065\textwidth}{0mm}
\begin{minipage}{0.47\textwidth}
\begin{align}
   \exists \vct y & \in\IR^n\setminus\{\vct0\} \quad\text{s.t.} \hspace{\tmpX} -\vct A^T\vct y\in C^\polar ,
\tag{D}
\end{align}
\end{minipage}
\\[2mm]
where $\vct A\in\IR^{n\times m}$ and $C^\polar = \{\vct z\in\IR^m\mid \langle \vct x,\vct z\rangle \leq 0 \text{ for all } \vct x\in C\}$ denotes the polar cone of~$C$ (the problem is to determine whether an input~$\vct A$ satisfies (P) or (D), which is almost surely, that is, for generic~$\vct A$, a strict alternative).
Special cases of interest in conic optimisation are when~$C$ is the non-negative orthant (LP), the second-order cone (QP), or the cone of positive semidefinite matrices (SDP). Other cases of interest, that include compressive sensing, are when~$C$ is the descent cone of a convex regularizer, in which case (the negation of)~(P) is referred to as a \emph{nullspace condition}.

It is obvious that the complexity of numerically solving the conic feasiblity problem does not only depend on the representation of $\mtx{A}$, but also on the intrinsic geometry of the problem. To quantify this, Renegar introduced a notion of condition, $\mathcal{R}_C(\mtx{A})$, that depends on both the geometry and the representation of $\mtx{A}$, and analysed in the context of linear programming the complexity of interior point methods for solving the feasibility problem~\cite{rene:94,rene:95a}. For the general conic feasibility problem, it was shown in~\cite{VRP:07} (see also~\cite[9.4]{Condition}) that the number of iterations of a barrier method for solving the feasibility problem is bounded by $O(\sqrt{\nu_C}\, \log(\nu_C \cdot \mathcal{R}_C(\mtx{A})))$, where $\nu_C$ is the barrier parameter of~$C$. Renegar's condition number $\mathcal{R}_C(\mtx{A})$ for descent cones of convex regularizers also appears uncredited in the error analysis of robust convex regularisation approaches to solving linear inverse problems~\cite{CRPW:12,FR:13}. 

The condition number $\mathcal{R}_C(\mtx{A})$ can become arbitrary large as $\mtx{A}$ approaches the boundary between (P) and (D). A probabilistic analysis of $\mathcal{R}_C(\mtx{A})$ and related quantities has been carried out in various settings; a cone-independent result from~\cite{AB:15} gives a bound of order $O(\log(n))$ for the expected logarithm of the related Grassmann condition number of a Gaussian matrix.

The theoretical bounds are somewhat out of sync with the observed performance. In fact, for large $m$ and $n$ and random $\mtx{A}$, it turns out that one of (P) or (D) will hold with overwhelming probability, while the other will hold with negligible probability, rendering the feasibility problem almost trivial; see~\cite{edge} and the references therein for a discussion of the associated phase transition phenomenon. 
One setting where this phenomenon has been embraced is compressive sensing~\cite{FR:13} and its generalisations. In this framework, the dual feasibility (D) is equivalent to a convex relaxation being successful.
From a more practical point of view, popular semidefinite programming solvers such as SDPT3 and Mosek rarely take more than 30 iterations on benchmark problems of various sizes~\cite{freund2007behavioral,dahlsemidefinite}; see also~\cite{roulet2015renegar} for some experiments in the context of compressive sensing.  

For symmetry reasons and to emphasise parallels to the matrix condition case, we consider in our analysis the following more general {\em biconic} convex feasibility problem with two nonzero closed convex cones $C\subseteq\IR^m$, $D\subseteq\IR^n$:
\\\def\tmpX{3mm}
\begin{minipage}{0.45\textwidth}
\begin{align}
   \exists \vct x & \in C\setminus\{\vct0\} \quad\text{s.t.} \hspace{\tmpX} \vct{Ax} \in D^\polar ,
\tag{P'}
\label{eq:(Pintro)}
\end{align}
\end{minipage}
\rule{0.065\textwidth}{0mm}
\begin{minipage}{0.47\textwidth}
\begin{align}
   \exists \vct y & \in D\setminus\{\vct0\} \quad\text{s.t.} \hspace{\tmpX} -\vct A^T\vct y\in C^\polar .
\tag{D'}
\label{eq:(Dintro)}
\end{align}
\end{minipage}
\\[2mm]
 (A generalisation of) Renegar's condition number is then defined by
\begin{equation}\label{eq:def-RCD}
  \RCD{A}{C}{D} = \min\bigg\{\frac{\|\vct A\|}{\sres{A}{C}{D}},\frac{\|\vct A\|}{\srestm{A}{D}{C}}\bigg\} ,
\end{equation}
where the smallest restricted singular value is defined as
  \[ \sres{A}{C}{D} = \min_{\vct x\in C\cap S^{m-1}} \|\Pi_D(\vct{Ax})\| ,\qquad \Pi_D(\vct y) = \min\{ \|\vct y-\vct y'\|\mid \vct y'\in D\} . \]
Note that in the case $C=\IR^m$, $D=\IR^n$ we recover the classical matrix condition number
  \[ \RCD{A}{\IR^m}{\IR^n}=\kappa(\vct A)=\|\vct A\|\,\|\vct A^\dagger\| , \]
  where $\mtx{A}^{\dagger}$ is the Moore-Penrose pseudoinverse of $\mtx{A}$.
In this unrestricted case a lot is known about the distribution of the condition number for Gaussian matrices but also for more general ensembles~\cite{CD:05,RV:09rect,BC:10,rudelson2010non}; see~\cite[Notes to Ch.~4]{Condition} for a concise discussion of this case. We only mention here that for Gaussian matrices it is known that if $m_k,n_k$ are such that $\lim_{k\to\infty} m_k/n_k = \gamma^2$, $0<\gamma<1$, then 
\begin{align*}
   \RCD{G}{\IR^m}{\IR^n} & =\kappa(\vct G)\to\frac{1+\gamma}{1-\gamma} \quad \text{almost surely} .
\end{align*}
In particular, in this asymptotic setting the expected condition number is of \emph{constant} order.

If $C$ and $D$ both have nonempty interior then the expectation of Renegar's condition number of Gaussian matrices $\RCD{G}{C}{D}$ is infinity. But if we consider families of cones such that their ``formats'' satisfy similar asymptotics as above, then we will show that Renegar's condition number does indeed have constant weak average-case complexity. We will describe next what we mean by similar format, after introducing a notion of dimension of a closed convex cone.

The statistical dimension of a closed convex cone $C\subseteq\IR^m$ may be characterized as
  \[ \sdim(C) = \underset{\vct g\sim N(\vct0,\vct I_m)}\E\big[ \|\Pi_C(\vct g)\|^2 \big] . \]
It coincides with the usual dimension if $C$ is the full space, $\sdim(\IR^m)=m$. We now consider families of cones $C_k\subseteq \R^{m_k}$, $D_k\subseteq \R^{n_k}$, and we assume that we have the asymptotics 
\begin{equation}\label{eq:assumpt-prop}
 \lim_{k\to \infty} \frac{\sdim(C_k)}{m_k} = \alpha^2, \quad \lim_{k\to \infty} \frac{\sdim(D_k)}{n_k} = \beta^2, \quad \lim_{k\to \infty} \frac{m_k}{n_k} = \gamma^2
\end{equation}
with $0<\alpha,\beta,\gamma\leq1$.
A typical example for this situation is $C_k\subseteq\IR^{m_k}$ self-dual and $D_k=\IR^{n_k}$, in which case $\sdim(C_k) = \frac{m_k}{2}$ and $\sdim(D_k)=n_k$, so that $\alpha=\frac{1}{\sqrt{2}}$ and $\beta=1$. We will assume that $\beta\neq\alpha\gamma$, and since $\Ren_{C,D}(\vct A)=\Ren_{D,C}(-\vct A^T)$, we may assume without loss of generality
\begin{equation}\label{eq:assumpt-beta>alpha.gamma}
   \beta > \alpha\,\gamma .
\end{equation}

Closely related, but conceptually different, to the statistical dimension is the Gaussian width of a cone, $w(C)=\underset{\vct g\sim N(\vct0,\vct I_m)}\E\big[ \|\Pi_C(\vct g)\| \big]$ (actually, it is the Gaussian width of the intersection with the unit ball $C\cap B^n$, but we adopt a purely conic point of view here). The squared Gaussian width $w^2(C)$ differs from the statistical dimension by at most one~\cite[Prop. 10.2]{edge}, $w^2(C)\leq\sdim(C)\leq w^2(C)+1$. In particular, the asymptotics~\eqref{eq:assumpt-prop} hold for the squared Gaussian width if and only if they hold for the statistical dimension.

\begin{thm}\label{thm:ren-cond-no.}
If $C_k\subseteq \R^{m_k}$, $D_k\subseteq \R^{n_k}$ are families of cones such that the dimensions satisfy the asymptotics~\eqref{eq:assumpt-prop} with $\beta>\alpha\gamma$, and $n_k\to\infty$ for $k\to \infty$, then there exist exceptional sets~$E_k$ with $\Prob\{\mtx{G}\in E_k\}\leq e^{-\Omega(n_k)}$, such that
\begin{align*}
   \E[\RCD{G}{C}{D}\mid \vct G\not\in E_k] & < u_k(C,D) , & u_k(C_k,D_k) \to \frac{1+\gamma}{\beta-\alpha\gamma} .
\end{align*}
In particular, if $n_k=\Omega(k)$, then the weak average-case complexity of $\Ren_{C,D}$ is constant. 
\end{thm}

The probabilistic analysis of Renegar's condition number for arbitrary cones has so far been confined to~\cite{AB:15} (though there has been considerable work on the linear programming case, see~\cite{Condition} and the references therein). These results were obtained following the classical tail estimate approach. The new approach of allowing to ignore exponentially small input sets, which loosens the requirements for the tail bounds so that for fixed dimensions we do not even require the bounds to go to zero, opens up new ways to directly exploit concentration of measure phenomena. We will do this through the use of Gordon's escape-through-the-mesh approach~\cite{Gordon,davidson2001local}.

\subsection{Relation to previous work}
Smoothed analysis~\cite{ST:06,buergICM} can also be seen as a way to deal with outliers. In its original Gaussian setting, one interpretation is that smoothed analysis corresponds to a worst-case analysis after Gaussian smoothing of the input space. When perturbing around ill-posed instances, this approach still suffers from the problem that too much weight is given to practically impossible events. An overview of previous work in the probabilistic analysis of numerical algorithms via condition numbers, that includes smoothed analysis and much more, can be found in~\cite{Condition}. 

The point of view of weak average-case analysis is not unusual in the setting of non-asymptotic random matrix theory~\cite{Vershynin2012} and its applications, such as compressive sensing and related fields~\cite{FR:13}. Rather than focussing on precise tail bounds $\Prob\{X> t\}$ that decay to zero for fixed $n$ as $t\to \infty$, and the associated expected values, the emphasis is on showing that a quantity of interest is confined to a certain region with overwhelming probability. 
In the setting of complexity theory this point of view seems to be new.

We would also like to point out that the idea of removing ``bad cases'' from a probabilistic analysis has featured in other probabilistic settings~\cite{venkataramanan2014gaussian,coja2014asymptotic,coja2013chasing,coja2012condensation,DMW:15}.

\subsection{Organisation of paper}
The remaining sections are devoted to the proofs of the three main results, Theorems~\ref{thm:con-cond-no.},~\ref{thm:powerit-main} and~\ref{thm:ren-cond-no.}. Section~\ref{sec:conic} deals with the easiest case, condition numbers inversely proportional to a set of ill-posed inputs. Here, a weak average-case analysis follows in a straight-forward way from a simple probabilistic observation on the effect of conditioning out large deviations from an expectation. Section~\ref{sec:powerit} deals with the power iteration. For convenience, we recreate Kostlan's derivation of the number of iterations needed to approach a dominant eigenvector in terms of the ratio of the largest singular values. The probabilistic analysis and proof of Theorem~\ref{thm:powerit-main} then follows from known results on random matrices from the Gaussian unitary ensemble. Finally, Section~\ref{sec:renegar} presents some basic facts about the biconic feasibility problem and Gordon's inequalities (discussed in more depth in~\cite{AL:14}), followed by the proof of Theorem~\ref{thm:ren-cond-no.}

\subsection{Acknowledgments}
The authors would like to thank Felipe Cucker for encouragement and useful comments on a preliminary draft.

\section{Conic condition numbers}\label{sec:conic}
We begin with some basic probabilistic observations about the effect of conditioning on the expectation of a quantity such as a condition number. We also include the weak average-case analysis of conic condition numbers, since it is a trivial consequence of Lemma~\ref{le:probexp}.
The first observation is that for a weak average-case analysis of a random variable $X$, while we are allowed to remove any small enough set from the probability space, we do best by removing a set of the form $\{X>t\}$.
Therefore, weak average-case analysis of $X$ is equivalent to an average-case analysis of a truncation $X_{\leq t} := X\cdot 1\{X\leq t\}$ for a suitable parameter $t>0$.

\begin{lemma}\label{le:isoprob}
Let $(\Omega,\mathcal{F},\mu)$ be a probability space and $X$ a random variable that is absolutely continuous with respect to $\mu$. Given $\varepsilon>0$, let $t$ be such that $\Prob\{X>t\}=\varepsilon$. Then for any measurable $S\subset \Omega$ with $\mu(S)\leq\varepsilon$,
\begin{equation*}
  \E \big[X\mid X\leq t\big] \leq \E\big[X\mid \overline{S}\big].
\end{equation*}
\end{lemma}

\begin{proof}
The conditional expectations $\E \big[X\mid X\leq t\big]$ are monotonically increasing in $t$, which is why we can assume without loss of generality that $\mu(S)=\veps$. Define $\ol S_{\leq t}=\{\vct{\omega}\in \ol S\mid X(\vct{\omega})\leq t\}$ and $\ol S_{>t}=\{\vct{\omega}\in \ol S\mid X(\vct{\omega})>t\}$. Then we have the disjoint decompositions $\ol S =\ol S_{\leq t}\cup\ol S_{>t}$ and
\begin{align*}
  \{\vct{\omega}\in \Omega \mid X(\vct{\omega})\leq t\} & = \ol S_{\leq t}\cup \{\vct{\omega}\in S \mid X(\vct{\omega})\leq t\} .
\end{align*}
Since $\Prob\{X\leq t\}=1-\varepsilon= \mu(\ol S)$, we get
\begin{equation}\label{eq:measure}
  \mu(\ol S_{> t}) = \mu(\{\vct{\omega}\in S\mid X(\vct{\omega})\leq t\}).
\end{equation}
For the conditional probabilities we obtain
\begin{align*}
\E \big[X\mid X\leq t\big] &= \frac{1}{1-\varepsilon}\bigg(\int_{\ol S_{\leq t}}X(\vct{\omega}) d\mu + \int_{\{S\mid X(\vct{\omega})\leq t\}}X(\omega) d\mu\bigg)
\\
&\leq \frac{1}{1-\varepsilon} \bigg(\int_{\ol S_{\leq t}}X(\vct{\omega}) d\mu+t\mu(\{S\mid X(\vct{\omega})\leq t\})\bigg)\\
&\stackrel{\eqref{eq:measure}}{=} \frac{1}{1-\varepsilon} \bigg(\int_{\ol S_{\leq t}}X(\vct{\omega}) d\mu+t\mu(\ol S_{> t})\bigg)
\\&
\leq \frac{1}{1-\varepsilon} \int_{\ol S} X(\vct{\omega}) d\mu = \E \big[X \mid \ol S\big]. \qedhere
\end{align*}
\end{proof}

The next lemma shows how to get finite conditional expectations from tail bounds of order~$t^{-1}$. This will be key for the weak average-case analysis of conic condition numbers and the power iteration method.

\begin{lemma}\label{le:probexp}
Let $X$ be a non-negative random variable on a probability space $(\Omega,\mathcal{F},\mu)$ such that for some $a>0$ and for all $t>a$,
\begin{equation}\label{eq:cheby}
  \Prob\{X>t\}\leq \frac{a}{t} .
\end{equation}
Then, for $t>a$,
\begin{equation*}
 \E\big[X \mid X\leq t\big] \leq \frac{a}{1-\frac{a}{t}}\Big(1-\log\Big(\frac{a}{t}\Big)\Big).
\end{equation*}
\end{lemma}

\begin{proof}
 The proof is a straight-forward calculation:
 \begin{align*}
  \E\big[X \mid X\leq t\big] &= \frac{1}{1-\Prob\{X>t\}}\int_{0}^t \Prob\{X>s\} \ ds\\
  &\leq \frac{1}{1-\frac{a}{t}}\left(a+\int_{a}^t \Prob\{X>s\} \ ds\right)
   \leq \frac{a}{1-\frac{a}{t}}(1+\log(t)-\log(a)). \qedhere
 \end{align*}
\end{proof}

Let $\CC(\vct{x})$ be a conic condition number as specified in Theorem~\ref{thm:con-cond-no.}. In~\cite[Thm.~1.1]{BCL:08} it is shown that for this kind of condition numbers,
\begin{equation}\label{eq:tubebound}
   \underset{\vct x\in B(\vct z,\sigma)}{\Prob}\{\CC(\vct x)>t\}  < \frac{13dn}{t\sigma} \quad \text{if } t\geq(1+2d)(n-1)/\sigma .
\end{equation}
Combining this powerful result with the simple Lemma~\ref{le:probexp} yields Theorem~\ref{thm:con-cond-no.}.

\begin{proof}[Proof of Theorem~\ref{thm:con-cond-no.}]
Set $a:=\frac{13 dn}{\sigma}$. Since $a=13dn/\sigma>(1+2d)(n-1)/\sigma$, the tail bound~\eqref{eq:tubebound} holds in particular for $t\geq a$. Setting $E_n=\{\vct{x}\mid \CC(\vct x)>ae^n\}$, we obtain
\begin{equation*}
 \underset{\vct x\in B(\vct z,\sigma)}{\Prob}\{\vct{x}\in E_n\} < e^{-n}.
\end{equation*}
Therefore,
\begin{align*}
 \underset{\vct x\in B(\vct z,\sigma)}{\E}\big[ \CC(\vct x)\mid \vct x\not\in E_n\big] &= \underset{\vct x\in B(\vct z,\sigma)}{\E}\big[ \CC(\vct x)\mid \CC(\vct{x})\leq  ae^n\big] 
 \stackrel{\text{Lem.}~\ref{le:probexp}}{\leq} \frac{a(n+1)}{1-e^{-n}} = \frac{13dn(n+1)}{(1-e^{-n})\sigma}.
\end{align*}
This completes the proof.
\end{proof}

\section{Power iteration}\label{sec:powerit}
We first review a bound on the number of iterations needed to get within a certain distance of the dominant eigenvector. 
We use the notation from Section~\ref{ssec:powerit}.
Recall that for $0<\alpha\leq\pi/2$ and $\vct{x}\in \C^{n}$, $\rho_{\alpha}(\vct{A},\vct{x})$ denotes the minimum number of iterations that bring $\mtx{A}^k\vct{x}$ into an $\alpha$-neighbourhood of $\vct{u}_1$, 
\begin{equation*}
 \rho_{\alpha}(\vct{A},\vct{x}) = \min\{k \mid d_R(\vct{p}_k,\vct{u}_1)\leq \alpha\},
\end{equation*}
and the expected value is defined as
\begin{equation*}
 \rho_{\alpha}(\mtx{A}) = \underset{\vct{x}\sim \Uniform(\CP^{n-1})}\Expect [\rho_{\alpha}(\mtx{A},\vct{x})],
\end{equation*}
where $\Uniform(\CP^{n-1})$ denotes the uniform distribution on complex projective space.
For a linear subspace $L\subseteq \C^n$, write $\Proj_L(\vct{x})=\mathrm{argmin}_{\vct{y}\in L} \|\vct{y}-\vct{x}\|$ for the projection of~$\vct{x}$ onto~$L$, and let $\Proj_{i}(\vct{x}):=\Proj_{\lspan\{\vct{u}_i\}}(\vct{x})$ be the projection onto the linear subspace spanned by the $i$-th eigenvector. The following result can be found in a slightly modified form in~\cite{K:88}; the proof is a variation of the well-known analysis of the power method in terms of the ratio of largest singular values, see for example~\cite[9.2]{wilkinson1965algebraic} or~\cite[7.3]{GoLoan}.
For convenience we provide a complete derivation of the starting point dependent results, since the proof in~\cite{K:88} is rather sketchy.

\begin{thm}[Kostlan~\cite{K:88}]\label{thm:conv-power}
Assume $\vct{x}\in \C^{n}$ is not orthogonal to the dominant eigenvector $\vct{u}_1$ of $\mtx{A}$, $|\langle \vct x,\vct u_1\rangle|<\frac{\pi}{2}$. Then 
\begin{equation*}
  \rho_{\alpha}(\mtx{A},\vct{x}) \geq \frac{\log\cot(\alpha)+\log \norm{\Proj_2(\vct{x})}-\log\norm{\Proj_1(\vct{x})}}{\log |\lambda_1|-\log |\lambda_2|}
\end{equation*}
and
\begin{equation*}
  \rho_{\alpha}(\mtx{A},\vct{x}) \leq \frac{\log\cot(\alpha)+\log \norm{\Proj_{\vct{u}_1^{\perp}}(\vct{x})}-\log\norm{\Proj_1(\vct{x})}}{\log |\lambda_1|-\log |\lambda_2|} .
\end{equation*}
The expected number of steps is bounded by
\begin{equation*}
 \frac{\log\cot(\alpha)}{\log |\lambda_1|-\log |\lambda_2|} \leq \rho_{\alpha}(\mtx{A}) \leq \frac{\frac{1}{2}(\log(n) + 2\, \frac{n-1}{n})+\max\{0,\log \cot\alpha\}}{\log |\lambda_1|-\log |\lambda_2|}.
\end{equation*}
\end{thm}

\begin{proof}
 Assume $\norm{\vct{x}}=1$ and 
 write $\vct{x} = a_1\vct{u}_1+\cdots +a_n\vct{u}_n$. Then, in particular, $a_1=\|\Proj_1(\vct{x})\|$, $a_2=\|\Proj_2(\vct{x})\|$, and $a_2\vct{u}_2+\cdots+a_n\vct{u}_n=\Proj_{\vct{u}_1^{\perp}}(\vct{x})$.
 Note that for any $k\geq 1$,
 \begin{equation*}
  \mtx{A}^k\vct{x} = a_1\lambda_1^k\vct{u}_1+\cdots +a_n\lambda_n^k\vct{u}_n .
 \end{equation*}
 Let $k=\rho_{\alpha}(\mtx{A},\vct{x})$ be the smallest integer such that 
 \begin{equation}\label{eq:b1}
  \cos(\alpha)\leq \frac{|\langle \mtx{A}^k\vct{x},\vct{u}_1\rangle|}{\|\mtx{A}^k\vct{x}\|}= \frac{|a_1\lambda_1^k|}{\|\mtx{A}^k\vct{x}\|} = \frac{\|\Proj_1(\vct{x})\| |\lambda_1|^k}{\|\mtx{A}^k\vct{x}\|}.
 \end{equation}
 From the identity $\sin^2(\alpha)+\cos^2(\alpha)=1$ we obtain
 \begin{equation}\label{eq:b2}
  \sin(\alpha) \geq \frac{\sqrt{a_2^2|\lambda_2|^{2k}+\cdots+a_n^2|\lambda_n|^{2k}}}{\|\mtx{A}^k\vct{x}\|}\geq \frac{|\lambda_2|^k\|\Proj_{2}(\vct{x})\|}{\|\mtx{A}^k\vct{x}\|}.
 \end{equation}
Putting~\eqref{eq:b1} and~\eqref{eq:b2} together, we get
\begin{equation*}
 \cot(\alpha) \leq \left(\frac{|\lambda_1|}{|\lambda_2|}\right)^k \frac{\|\Proj_1(\vct{x})\|}{\|\Proj_2(\vct{x})\|} \ \Rightarrow \ k \geq \frac{\log\cot(\alpha)+\log \norm{\Proj_2(\vct{x})}-\log\norm{\Proj_1(\vct{x})}}{\log |\lambda_1|-\log |\lambda_2|}.
\end{equation*}
This shows the first inequality. For the second inequality, let $k<\rho_\alpha(\mtx{A},\vct{x})$. Then after $k$ iterations we are still outside an $\alpha$ neighbourhood of $\vct{u}_1$, so that
\begin{equation}\label{eq:c1}
  \cos(\alpha)\geq \frac{|\langle \mtx{A}^k\vct{x},\vct{u}_1\rangle|}{\|\mtx{A}^k\vct{x}\|}= \frac{|a_1\lambda_1^k|}{\|\mtx{A}^k\vct{x}\|}= \frac{\|\Proj_1(\vct{x})\| |\lambda_1|^k}{\|\mtx{A}^k\vct{x}\|}.
 \end{equation}
 Similar as in~\eqref{eq:b2}, we get
 \begin{equation}\label{eq:c2}
  \sin(\alpha) \leq \frac{\sqrt{a_2^2|\lambda_2|^{2k}+\cdots+a_n^2|\lambda_n|^{2k}}}{\|\mtx{A}^k\vct{x}\|}\leq \frac{|\lambda_2|^k\|\Proj_{\vct{u}_1^{\perp}}(\vct{x})\|}{\|\mtx{A}^k\vct{x}\|}.
 \end{equation}
 Combining~\eqref{eq:c1} and~\eqref{eq:c2} gives
 \begin{equation*}
  \cot(\alpha) \geq \left(\frac{|\lambda_1|}{|\lambda_{2}|}\right)^k \frac{\|\Proj_1(\vct{x})\|}{\|\Proj_{\vct{u}_1^{\perp}}(\vct{x})\|} \ \Rightarrow \ k \leq \frac{\log\cot(\alpha)+\log \norm{\Proj_{\vct{u}_1^{\perp}}(\vct{x})}-\log\norm{\Proj_1(\vct{x})}}{\log |\lambda_1|-\log |\lambda_2|},
 \end{equation*}
which establishes the second inequality.

For the expected values of the bounds we refer to Kostlan's work~\cite{K:88}.
\end{proof}

\subsection{Weak average-case analysis of power iteration}

Theorem~\ref{thm:conv-power} implies that the relevant probability in the average-case analysis of power iteration is
\begin{equation*}
 \Prob\bigg\{\frac{1}{\log |\lambda_1|-\log|\lambda_2|}> x\bigg\} = \Prob\bigg\{\frac{|\lambda_1|}{|\lambda_2|}<e^{\frac{1}{x}}\bigg\}.
\end{equation*}
Let $\lambda_{\mx}$, $\lambda_{\nx}$, $\lambda_{\mn}$ and $\lambda_{\lx}$ denote the largest, second largest, smallest and second smallest eigenvalue of a GUE matrix $\mtx{H}$, respectively.
Keep in mind that this refers to the eigenvalues themselves, not their absolute values, which underlie the ordering $|\lambda_1|\geq\dots\geq|\lambda_n|$.
Before proceeding, we collect some known results on the distribution of the eigenvalues of a GUE matrix. To ease notation, in what follows, $C$ and $c$ are used for absolute constants which may change from line to line.

\begin{enumerate}
  \item The probability that all eigenvalues are positive (or all are negative) is exponentially small~\cite{DM:08},
    \begin{equation}\label{eq:indefinite}
     \Prob\{\lambda_{\mn}>0\}\leq Ce^{-c n^2}
    \end{equation}
    for some constants $C$ and $c$.
  \item The largest eigenvalue of a GUE matrix satisfies
 \begin{equation}\label{eq:eig-tail}
  \Prob\{\lambda_{\mx}\geq \sqrt{n}(2+\veps)\}\leq Ce^{-cn\veps^{3/2}}
 \end{equation}
for some constants $C$ and $c$ (e.g., ~\cite[Prop. 2.1]{ledoux2007deviation}, note the slightly different form). We will use this with $\veps=1$.
  \item The smallest gap $\delta_{\mn}$ between consecutive eigenvalues satisfies
\begin{equation}\label{eq:small-gap}
 \Prob\{\delta_{\mn}\leq \delta/\sqrt{n}\}\leq n\delta^3+e^{-cn}
\end{equation}
for some constant $c$. This is from~\cite{erdHos2010wegner}, as stated in~\cite[(2)]{NTV:15}.
\end{enumerate}

\begin{proof}[Proof of Theorem~\ref{thm:powerit-main}.]
Assume $x\geq 1$.
Throughout the proof, we will use the fact that $-\lambda_{\mn}$ and $\lambda_{\mx}$ have the same distribution, as well as $-\lambda_{\lx}$ and $\lambda_{\nx}$.
For the relationship between the singular values and the eigenvalues we have the following two cases:
\begin{enumerate}
    \item The largest and second largest singular values come from the two largest or from the two smallest eigenvalues, that is, $\lambda_1=\lambda_\mx$ and $\lambda_2=\lambda_\nx$, or $\lambda_1=\lambda_\mn$ and $\lambda_2=\lambda_\lx$.
    \item The largest and second largest singular values come from the largest (positive) and smallest (negative) eigenvalues, that is, $\lambda_1=\lambda_\mx$ and $\lambda_2=\lambda_\mn$, or $\lambda_1=\lambda_\mn$ and $\lambda_2=\lambda_\mx$.
  \end{enumerate}
Using the union bound we may consider these cases separately:
\begin{align*}
   \Prob\bigg\{\frac{|\lambda_1|}{|\lambda_2|}<e^{\frac{1}{x}}\bigg\} & \leq \Prob\bigg\{\frac{|\lambda_{\mx}|}{|\lambda_{\nx}|}<e^{\frac{1}{x}} \bigg\}+\Prob\bigg\{\frac{|\lambda_{\mn}|}{|\lambda_{\lx}|}<e^{\frac{1}{x}} \bigg\}
\\ & \qquad +\Prob\bigg\{e^{-\frac{1}{x}}<\frac{|\lambda_{\mx}|}{|\lambda_{\mn}|}<e^{\frac{1}{x}}\bigg\}
\end{align*}
Clearly, the first two terms on the right-hand side are subject to the same analysis.

\smallskip

\noindent{\bf Adjacent eigenvalues.}
We consider the case where the largest and second largest singular values are the largest and second largest (positive) eigenvalues of the matrix. Note that for $x\geq 1$ we have $e^{1/x}-1<2/x$, which will be used in the sequel. We will also generically use $C$ and $c$ for constants that may vary within one derivation.
We have for $x\geq1$,
\begin{align*}
   \Prob\bigg\{\frac{|\lambda_{\mx}|}{|\lambda_{\nx}|}<e^{\frac{1}{x}}\bigg\} & \leq 
    \Prob\bigg\{\frac{\lambda_{\mx}}{\lambda_{\nx}}<e^{\frac{1}{x}}\bigg\} = \Prob\bigg\{\frac{\lambda_{\mx}-\lambda_{\nx}}{\lambda_{\nx}}<e^{1/x}-1\bigg\}
\\ & \leq \Prob\bigg\{\frac{\delta_{\mathrm{min}}}{\lambda_{\mx}}<\frac{2}{x}\bigg\}
\stackrel{\eqref{eq:eig-tail}, \veps=1}{\leq} \Prob\bigg\{ \delta_{\mn}<\frac{6n}{x\sqrt{n}}\bigg\}+Ce^{-cn}
\\ & \stackrel{\eqref{eq:small-gap}}{<} 216\frac{n^4}{x^3}+Ce^{-cn}.
\end{align*}
For $x>6n^{4/3}$ and $n$ large enough, the bound becomes non-trivial.

\smallskip

\noindent{\bf Extreme eigenvalues.} For the ratio of the eigenvalues at the edge we have for $x\geq 1$,
\begin{align*}
 \Prob & \bigg\{e^{-\frac{1}{x}} <\frac{|\lambda_{\mx}|}{|\lambda_{\mn}|}<e^{\frac{1}{x}}\bigg\} \leq \Prob\bigg\{e^{-\frac{1}{x}}<\frac{\lambda_{\mx}}{-\lambda_{\mn}}<e^{\frac{1}{x}}\bigg\} + 2\Prob\{\lambda_{\mn}>0\}
 \\ & \stackrel{\eqref{eq:indefinite}}{\leq}
 \Prob\bigg\{\frac{\lambda_{\mx}}{-\lambda_{\mn}}<e^{\frac{1}{x}}\bigg\} - \Prob\bigg\{\frac{\lambda_{\mx}}{-\lambda_{\mn}}\leq e^{-\frac{1}{x}}\bigg\}+Ce^{-cn^2}
 \\ & = \Prob\bigg\{\frac{\lambda_{\mx}}{-\lambda_{\mn}}<e^{\frac{1}{x}}\bigg\} - \bigg(1-\Prob\bigg\{\frac{\lambda_{\mx}}{-\lambda_{\mn}}> e^{-\frac{1}{x}}\bigg\}\bigg)+Ce^{-cn^2}
 \\ & = 2\Prob\bigg\{\frac{\lambda_{\mx}}{-\lambda_{\mn}}<e^{\frac{1}{x}}\bigg\}-1 + Ce^{-cn^2}
 \\ & = 2\Prob\bigg\{\frac{\lambda_{\mx}+\lambda_{\mn}}{-\lambda_{\mn}}<e^{\frac{1}{x}}-1\bigg\}-1+ Ce^{-cn^2}
 \\ & \leq 2\Prob\bigg\{\frac{\lambda_{\mx}+\lambda_{\mn}}{-\lambda_{\mn}}<\frac{2}{x}\bigg\}-1+ Ce^{-cn^2}
\\ & \stackrel{\eqref{eq:eig-tail}, \veps=1}{\leq} 2\Prob\bigg\{\lambda_{\mx}+\lambda_{\mn}<\frac{6\sqrt{n}}{x}\bigg\}-1+ Ce^{-cn^2} .
\end{align*}
Since $\lambda_{\mx}$ and $-\lambda_{\mn}$ are equally distributed, the difference $\lambda_{\mx}-(-\lambda_{\mn})$ is symmetric around the origin, so that we obtain
  \[ 2\Prob\bigg\{\lambda_{\mx}+\lambda_{\mn}<\frac{6\sqrt{n}}{x}\bigg\}-1 = \Prob\bigg\{-\frac{6}{x}<\frac{\lambda_{\mx}}{\sqrt{n}}+\frac{\lambda_{\mn}}{\sqrt{n}}<\frac{6}{x}\bigg\} . \]
Now consider the normalized variables 
\begin{equation*}
 \tilde{\lambda}_{\mx} := n^{2/3}\bigg(\frac{\lambda_{\mx}}{\sqrt{n}}-2\bigg), \qquad \tilde{\lambda}_{\mn} := n^{2/3}\bigg(\frac{\lambda_{\mx}}{\sqrt{n}}+2\bigg).
\end{equation*}
The maximum $M$ of the density of $\tilde{\lambda}_{\mx}+\tilde{\lambda}_{\mn}$ is bounded, as can be deduced directly from the joint distribution of the eigenvalues of the GUE ensemble. 
We thus obtain
\begin{equation*}
 \Prob\bigg\{-\frac{6n^{2/3}}{x}<\tilde{\lambda}_{\mx}+\tilde{\lambda}_{\mn}<\frac{6n^{2/3}}{x}\bigg\} \leq \frac{12 M n^{2/3}}{x}.
\end{equation*}

In summary, we have the bound 
\begin{equation}\label{eq:final-ext-eig}
\Prob\bigg\{e^{-\frac{1}{x}}<\frac{|\lambda_{\mx}|}{|\lambda_{\mn}|}\leq e^{\frac{1}{x}}\bigg\} \leq \frac{12 M n^{2/3}}{x}+Ce^{-cn} .
\end{equation}

\smallskip

\noindent{\bf Combined bound.} 
Putting the bounds together, we get
\begin{equation*}
 \Prob\bigg\{\frac{1}{\log|\lambda_1|-\log|\lambda_2|}>x\bigg\} < \frac{12Mn^{2/3}}{x}+\frac{432 n^4}{x^3}+Ce^{-cn}
\end{equation*}
for some constant $K$.
This bound is non-trivial for $x>Kn^{4/3}$ for a suitable constant~$K$. Set $x_0=Kn^{4/3}e^{cn}$ and
\begin{equation*}
 E_n := \bigg\{\mtx{H} \;\Big|\; \frac{1}{\log|\lambda_1|-\log|\lambda_2|}>x_0\bigg\}.
\end{equation*}
Then
\begin{equation*}
 \Prob\left\{\mtx{H}\in E_n\right\} < Ce^{-cn}
\end{equation*}
for suitable absolute constants $C$ and $c$.
We can apply a variation of the proof of Lemma~\ref{le:probexp} (with an adjustment to take into account the $x^{-3}$ and the exponential term), with $a=Kn^{4/3}$ and $t=x_0$, to get
\begin{equation*}
 \Expect\Big[(\log|\lambda_1|-\log|\lambda_2|)^{-1} \,\Big|\, E_n\Big] < \frac{Kn^2}{1-e^{-cn}}
\end{equation*}
for some absolute constant $K$.
Incorporating the numerator from the bound on the power iteration gives the desired result.
\end{proof}

\begin{remark}
 With a more careful analysis, the constants in the bounds can be explicitly estimated, but we haven't done so as the first aim was to establish a polynomial bound.
\end{remark}

\begin{remark}
There are potential alternative ways to derive the results from this section.
 One such approach is based on the joint distribution of the singular values, as studied by Edelman and La Croix~\cite{ELC:14}. There, it is shown that the singular values of a GUE matrix can be characterized as the union of the singular values of two independent bidiagonal matrices with $\chi$-distributed entries. By an interlacing property it appears that, with high probability, the largest two singular values correspond to the largest singular values of each of these matrices. 
 In a different direction, by examining the analysis in~\cite{bianchi2010asymptotic,bornemann2010asymptotic} (for asymptotic independence of the extreme eigenvalues) and ~\cite{johnstone2012fast} (for the rate of convergence to Tracy-Widom) might also hope to carry out the above analysis via approximation by the Tracy-Widom distribution. 
\end{remark}

\section{Renegar's condition number}\label{sec:renegar}

Note that we have a complete symmetry between~(P') and~(D') via the exchange of $\vct A$ by $-\vct A^T$ and by a swap between $C$ and $D$.
We denote the set of primal feasible instances and the set of dual feasible instances by
\begin{align*}
   \P(C,D) & := \{\vct A\in\IR^{n\times m}\mid \text{(P') is feasible}\} , & \D(C,D) & := \{\vct A\in\IR^{n\times m}\mid \text{(D') is feasible}\},
\end{align*}
and we call
  \[ \Sigma(C,D) := \P(C,D)\cap \D(C,D) \]
the set of \emph{ill-posed} inputs. Indeed, it can be shown that $\P(C,D)$ and $\D(C,D)$ are both closed, the union of $\P(C,D)$ and $\D(C,D)$ is the whole input space~$\IR^{n\times m}$ unless $C=D=\IR^m$, and the probability that a Gaussian matrix lies in $\Sigma(C,D)$ is zero~\cite[Sec.~2.2]{AL:14}.

As for the relation of this generalized feasibility problem to conically restricted linear operators, observe that $\vct A\in\P(C,D)$ if and only if $\sres{A}{C}{D}=0$, and $\vct A\in\D(C,D)$ if and only if $\srestm{A}{D}{C}=0$. Moreover, we have~\cite[Sec.~2.2]{AL:14}
\begin{align}\label{eq:sres-dist}
  \dist(\vct A,\P) & = \sres{A}{C}{D} , & \dist(\vct A,\D) & = \srestm{A}{D}{C} ,
\end{align}
where $\dist$ denotes the Euclidean distance.

The following theorem is the consequence of a classic result by Gordon~\cite{Gordon}. For various derivations, see~\cite{davidson2001local,FR:13}. The form stated here, for a map between two cones, is from~\cite{AL:14}. Here, we use the Gaussian width $w(C)=\underset{\vct g\sim N(\vct0,\vct I_m)}\E\big[ \|\Pi_C(\vct g)\| \big]$.

\begin{thm}
Let $C\subseteq\IR^m$, $D\subseteq\IR^n$ closed convex cones, and let $\vct G\in\IR^{n\times m}$ be a Gaussian matrix. Then for $\lambda\geq0$,
\begin{align}\label{eq:gordon}
\nonumber
   \Prob\big\{\nres{G}{C}{D}\geq \sdimw(C)+\sdimw(D)+\lambda\big\} & \leq e^{-\frac{\lambda^2}{2}}, 
\\
\Prob\big\{\sres{G}{C}{D}\leq \sdimw(D)-\sdimw(C)-\lambda\big\} & \leq e^{-\frac{\lambda^2}{2}}.
\end{align}
\end{thm}

Before getting to the proof of the asymptotic statement in Theorem~\ref{thm:ren-cond-no.} we use Gordon's estimate~\eqref{eq:gordon} for the following finite-dimensional bound.

\begin{proposition}\label{prop:keybound}
Let $C\subseteq\IR^m$, $D\subseteq\IR^n$ closed convex cones with $\sdimw(D)-\sdimw(C)>2\sqrt{2}$, and let $\vct G\in\IR^{n\times m}$ be a Gaussian matrix. Then there exists a set $E\subseteq\R^{n\times m}$ with $\Prob\{\vct G\in E\}<\veps := 2e^{-(\sdimw(D)-\sdimw(C))^2/8}$, such that
\begin{equation}\label{eq:keybound}
 \E[\RCD{G}{C}{D}\mid \vct G\not\in E] < \frac{1}{1-\veps}\, \bigg(\frac{\sdimw(\R^n)+\sdimw(\R^m)}{\sdimw(D)-\sdimw(C)} + 4 \int_{\frac{b}{a}}^{\frac{a+b}{2a}} \frac{e^{-\frac{(as-b)^2}{2}}}{(1-s)^2}\,ds\bigg) ,
\end{equation}
where $a = \frac{1}{2}\big(\sdimw(\R^m)+\sdimw(\R^n)+\sdimw(D)-\sdimw(C)\big)$, $b = \frac{1}{2}\big(\sdimw(\R^m)+\sdimw(\R^n)-(\sdimw(D)-\sdimw(C)\big)$.
\end{proposition}

\begin{proof}
To ease the notation, set $w_m=w(\R^m)$, $w_n=w(\R^n)$. Note that $a+b = w_m+w_n$ and $a-b = w(D)-w(C)$. Equations~\eqref{eq:gordon} and the union bound give for $0\leq\lambda<\sdimw(D)-\sdimw(C)$,
\begin{equation*}
 \Prob\bigg\{\frac{\|\vct G\|}{\sres{G}{C}{D}}\geq \frac{w_m+w_n+\lambda}{\sdimw(D)-\sdimw(C)-\lambda} \bigg\} \leq 2e^{-\frac{\lambda^2}{2}}.
\end{equation*}
We introduce $t=\frac{w_m+w_n+\lambda}{\sdimw(D)-\sdimw(C)-\lambda}$, for which $\frac{t-1}{t+1}=\frac{b+\lambda}{a}$, so that we can rewrite the bound as
\begin{equation*}
 \Prob\bigg\{\frac{\|\vct G\|}{\sres{G}{C}{D}}\geq t\bigg\}\leq 2\exp\Big(-\frac{1}{2}\Big( a\,  \frac{t-1}{t+1} - b \Big)^2 \Big) .
\end{equation*}
This holds for $\frac{t-1}{t+1}\geq\frac{b}{a}$, or equivalently, $t\geq\frac{a+b}{a-b}=\frac{w_m+w_n}{w(D)-w(C)}$. In particular, for $t_\veps:=2\frac{a+b}{a-b}+1$, that is, $\frac{t_\veps-1}{t_\veps+1}=\frac{a+b}{2a}$, we have
\begin{equation*}
 \Prob\bigg\{\frac{\|\vct G\|}{\sres{G}{C}{D}}\geq t_\veps\bigg\}\leq 2\exp\Big(-\frac{(a-b)^2}{8}\Big) = \veps .
\end{equation*}
Defining the exceptional set $E:=\{\vct G\mid \frac{\|\vct G\|}{\sres{G}{C}{D}}\geq t_\veps\}$, we have, in particular, that $\sres{G}{C}{D}>0$ for $\vct{G}\not\in E$. By~\eqref{eq:def-RCD} and~\eqref{eq:sres-dist} this implies that the condition number of $\vct G\not\in E$ is given by $\RCD{G}{C}{D} = \frac{\|\vct G\|}{\sres{G}{C}{D}}$. For the conditional expectation we thus obtain
\begin{align*}
   \E[\RCD{G}{C}{D}\mid \vct G\not\in E] & \leq \frac{1}{1-\veps}\, \int_0^{t_\veps} \Prob\Big\{\RCD{G}{C}{D}\geq t\text{ and } \vct G\not\in E\Big\}\,dt
\\ & \leq \frac{1}{1-\veps}\, \bigg( \frac{a+b}{a-b} + 2 \int_{\frac{a+b}{a-b}}^{t_\veps} \exp\Big(-\frac{1}{2}\Big( a\,  \frac{t-1}{t+1} - b \Big)^2 \Big) \,dt\bigg).
\end{align*}
Substituting $s=\frac{t-1}{t+1}$, the expression reads as
\begin{equation*}
 \E[\RCD{G}{C}{D}\mid \vct G\not\in E] \leq \frac{1}{1-\veps}\, \bigg(\frac{a+b}{a-b} + 4 \int_{\frac{b}{a}}^{\frac{a+b}{2a}} \frac{e^{-\frac{(as-b)^2}{2}}}{(1-s)^2}\,ds\bigg) ,
\end{equation*}
which was to be shown.
\end{proof}

For the proof of Theorem~\ref{thm:ren-cond-no.} it remains to estimate the integral in~\eqref{eq:keybound}. Since the boundaries of integration are both in the interval $(0,1)$ we could get an easy bound by computing the maximum of the integrand and bounding the integral accordingly. Instead, we will use Laplace's method~\cite{bender1999advanced} to show that the integral goes to zero in the assumed asymptotic setting.

\begin{proof}[Proof of Theorem~\ref{thm:ren-cond-no.}]
Recall that $C_k\subseteq \R^{m_k}$, $D_k\subseteq \R^{n_k}$ are families of cones such that the dimensions satisfy $\lim_{k\to \infty} \sdim(C_k)/m_k = \alpha^2$, $\lim_{k\to \infty} \sdim(D_k)/n_k = \beta^2$, and $\lim_{k\to \infty} m_k/n_k = \gamma^2$, with $0<\alpha,\beta,\gamma\leq1$ and $\beta > \alpha\,\gamma$. To see that Proposition~\ref{prop:keybound}, applied to these sequences of cones, does indeed give the claimed asymptotics in Theorem~\ref{thm:ren-cond-no.}, note first that the size of the exceptional set is bounded by $e^{-\Omega(n_k)}$, as
\begin{align*}
   \frac{(\sdimw(D)-\sdimw(C))^2}{n_k} & = \frac{\sdimw(D)^2}{n_k} - 2\frac{\sdimw(D)\sdimw(C)}{n_k} + \frac{\sdimw(C)^2}{n_k}
\\ & \to \beta^2-2\alpha\beta\gamma+\alpha\gamma = \beta(\beta-\alpha\gamma) + \alpha\gamma(1-\beta) > 0 ,
\end{align*}
where we used $\sdimw(C)^2\leq\sdim(C)\leq\sdimw(C)^2+1$ and the assumption $n_k\to\infty$ for $k\to\infty$. So the size of the exceptional set is small enough.

Hence, all that remains to show is that the integral in the bound~\eqref{eq:keybound} goes to zero for $k\to\infty$. More precisely, if we define as in Proposition~\ref{prop:keybound},
\begin{align*}
   a_k & = \tfrac{\sdimw(\R^{m_k})+\sdimw(\R^{n_k})}{2}+\tfrac{\sdimw(D_k)-\sdimw(C_k)}{2} ,
 & b_k & = \tfrac{\sdimw(\R^{m_k})+\sdimw(\R^{n_k})}{2}-\tfrac{\sdimw(D_k)-\sdimw(C_k)}{2} ,
\end{align*}
then the decisive term from~\eqref{eq:keybound} is given by
  \[ R_k := \int_{\frac{b_k}{a_k}}^{\frac{a_k+b_k}{2a_k}} g(s) \exp\big(n_k f_k(s)\big)\,ds , \]
where $g(s)=(1-s)^{-2}$ and $f_k(s)=-\frac{(a_ks-b_k)^2}{2n_k}$, which we need to show that it goes to zero.

The lower integral bound is not a problem so that we use the trivial lower bound $b_k/a_k\geq0$. For the upper integral bound we have the asymptotics
\begin{equation*}
  \frac{a_k+b_k}{2a_k} = \frac{\sdimw(\R^{m_k})+\sdimw(\R^{n_k})}{\sdimw(\R^{m_k})+\sdimw(\R^{n_k})+\sdimw(D_k)-\sdimw(C_k)} \to 
  u := \frac{1}{1+\frac{\beta+\alpha\gamma}{1+\gamma}} <1 ,
\end{equation*}
and for the functions $f_k$ we notice that
\begin{align*}
   f_k(s) & = -\frac{\Big(\big(\frac{\sdimw(\R^{m_k})}{n_k}+\frac{\sdimw(\R^{n_k})}{n_k}\big)(s-1)+\big(\frac{\sdimw(D_k)}{n_k}-\frac{\sdimw(C_k)}{n_k}\big)(s+1)\Big)^2}{8}
\\ & \to f(s) := -\frac{((\gamma+1)(s-1)+(\beta-\alpha\gamma)(s+1))^2}{8} .
\end{align*}
So if we define
  \[ R := \int_0^u g(s)\,e^{nf(s)}\,ds , \]
we see that all that remains is to show that $R\to0$ for $n\to\infty$. Now, $f(s)$ can be written in the form
  \[ f(s) = -\frac{(cs-d)^2}{8} ,\qquad c=\gamma+1+\beta-\alpha\gamma ,\quad d=\gamma+1-(\beta-\alpha\gamma) , \]
and we can estimate $R\leq g(u)\, \int_0^u e^{-n(cs-d)^2/8}\,ds$. Since
\begin{align*}
   0 < \frac{d}{c} & = \frac{\gamma+1-(\beta-\alpha\gamma)}{\gamma+1+\beta-\alpha\gamma} < \frac{\gamma+1}{\gamma+1+\beta-\alpha\gamma} = u ,
\end{align*}
Laplace's method yields $\int_0^u e^{-n(cs-d)^2/8}\,ds=O(n^{-1/2})$. This shows $R\to0$ for $n\to\infty$ and thus finishes the proof.
\end{proof}

\bibliographystyle{alpha}
\bibliography{wacco}

\end{document}